\documentclass{amsart}

\usepackage{graphicx}
\usepackage{amssymb}
\usepackage{float}

\newtheorem{theorem}{Theorem}[section]
\newtheorem{lemma}[theorem]{Lemma}
\newtheorem{corollary}{Corollary}[theorem]
\newtheorem{proposition}[theorem]{Proposition}

\theoremstyle{definition}

\theoremstyle{remark}
\newtheorem{remark}[theorem]{Remark}

\numberwithin{equation}{section}

\newcommand{\PP}{\mathcal{P}}

\newcommand{\Q}{{\mathbb Q}}

\newcommand{\N}{{\mathbb N}}

\newcommand{\li}{\textnormal{li}}

\newcommand{\QQ}{{\mathcal Q}}

\begin{document}

\title[The Erd\H os conjecture for primitive sets]{The Erd\H os conjecture for primitive sets}

\author{Jared Duker Lichtman}
\address{Department of Mathematics, Dartmouth College, Hanover, NH 03755}

\email{jdl.18@dartmouth.edu}
\email{jared.d.lichtman@gmail.com}

\author{Carl Pomerance}
\address{Department of Mathematics, Dartmouth College, Hanover, NH 03755}
\email{carl.pomerance@dartmouth.edu}


\date{June 30, 2018.}



\begin{abstract}
A subset of the integers larger than 1 is {\it primitive} if no member divides another.
Erd\H os proved in 1935 that the sum of $1/(a\log a)$ for $a$ running over a primitive set $A$
is universally bounded over all choices for $A$.  In 1988 he asked if this universal bound is
attained for the set of prime numbers.  In this paper we make some progress on several fronts,
and show a connection to certain prime number ``races" such as the race between $\pi(x)$
and $\li(x)$.
\end{abstract}

\maketitle


\section{Introduction}

A set of positive integers $>1$ is called {\bf primitive} if no element divides any other (for convenience, we exclude the singleton set $\{1\}$).   There are
a number of interesting and sometimes unexpected theorems about primitive sets.
After Besicovitch \cite{besicovitch}, we know that the upper asymptotic density of a primitive set can be arbitrarily close to $1/2$, whereas the lower asymptotic density is always $0$.  Using the fact that
if a primitive set has a finite reciprocal sum, then 
the set of multiples of members of the set has an asymptotic density, Erd\H os gave an elementary proof that the
set of nondeficient numbers (i.e., $\sigma(n)/n\ge2$, where $\sigma$ is the sum-of-divisors
function) has an asymptotic density.  Though the reciprocal sum of a primitive set can
possibly diverge, Erd\H os \cite{erdos35} showed that for a primitive set $A$,
$$
\sum_{a\in A}\frac1{a\log a}<\infty.
$$
In fact, the proof shows that these sums are uniformly bounded as $A$ varies over
primitive sets.

Some years later in a 1988 seminar in Limoges, Erd\H os suggested that in fact we always have
\begin{equation}
\label{eq:conj}
f(A):=\sum_{a\in A}\frac1{a\log a}\le\sum_{p\in\PP}\frac1{p\log p},
\end{equation}
where $\PP$ is the set of prime numbers.  The assertion \eqref{eq:conj} is now known
as the Erd\H os conjecture for primitive sets.  

In 1991, Zhang \cite{zhang1} proved the
Erd\H os conjecture for primitive sets $A$ with no member having more than 4 prime factors
(counted with multiplicity).

After Cohen \cite{cohen}, we have
\begin{equation}
\label{eq:cohen}
    C: = \sum_{p\in \PP}\frac1{p\log p} = 1.63661632336\ldots\,,
\end{equation}
the sum over primes in \eqref{eq:conj}.  
Using the original Erd\H os argument in \cite{erdos35}, Erd\H os and Zhang showed that
$f(A)<2.886$ for a primitive set $A$, which was later improved by Robin to $2.77$.  These unpublished
estimates are reported in Erd\H os--Zhang \cite{ez} who used another method to show that
 $f(A)<1.84$.   Shortly after, Clark \cite{clark} claimed that 
$f(A)\le e^\gamma=1.781072\dots$\,.  However, his brief argument appears to be
incomplete.  

Our principal results are the following.
\begin{theorem}\label{thm:egamma}
For any primitive set $A$ we have $f(A) < e^\gamma$.
\end{theorem}
\begin{theorem}
\label{thm:no8s}
For any primitive set $A$ with no element divisible by $8$, we have $f(A)<C+2.37\times10^{-7}$.
\end{theorem}

Say a prime $p$ is
{\bf Erd\H os strong} if for any primitive set $A$ with the property that each element of $A$ has least prime
factor $p$, we have $f(A)\le 1/(p\log p)$.
We conjecture that every prime is Erd\H os strong.  Note that 
the Erd\H os conjecture \eqref{eq:conj} would immediately follow, though it is not clear that the Erd\H os conjecture implies our conjecture.  Just proving our conjecture for the case of $p=2$ would give the inequality in 
Theorem \ref{thm:no8s} for all primitive sets $A$.
Currently the best we can do for a primitive set $A$ of even numbers is that
$f(A)<e^\gamma/2$, see Proposition \ref{lem:erdos} below.

For part of the next result, we assume the Riemann hypothesis (RH) and the Linear Independence hypothesis (LI), which asserts that the sequence of numbers $\gamma_n>0$ such that $\zeta(\tfrac{1}{2}+i\gamma_n)=0$ is linearly independent over $\Q$.

\begin{theorem}
\label{thm:race}
Unconditionally, all of the odd
primes among the first $10^8$ primes are Erd\H os strong.
Assuming RH and LI, 
the Erd\H os strong primes have  relative lower logarithmic
density $>0.995$.
\end{theorem}
The proof depends strongly on a recent result of Lamzouri \cite{lamz} who was interested in
the ``Mertens race" between $\prod_{p\le x}(1-1/p)$ and $1/(e^\gamma \log x)$.

For a primitive set $A$, let $\PP(A)$ denote the support of $A$, i.e., the 
set of prime numbers that divide some member of $A$.  It is clear that the
Erd\H os conjecture \eqref{eq:conj} is equivalent to the same assertion where
the prime sum is over $\PP(A)$.
\begin{theorem}
\label{thm:support}
If $A$ is a primitive set with $\PP(A)\subset[3,\exp(10^6)]$, then
$$
f(A)\le\sum_{p\in\PP(A)}\frac{1}{p\log p}.
$$
\end{theorem}
If some primitive set $A$ of odd numbers exists with $f(A)>\sum_{p\in\PP(A)}1/(p\log p)$,
Theorem \ref{thm:support} suggests that it will be very difficult indeed to give a concrete example!

For a positive integer $n$, let $\Omega(n)$ denote the number of prime factors of $n$
counted with multiplicity.  Let $\N_k$ denote the set of integers $n$ with $\Omega(n)=k$.
Zhang \cite{zhang2} proved a result that implies $f(\N_k)< f(\N_1)$ for each $k\ge2$, so that the Erd\H os conjecture holds for the primitive sets $\N_k$.  More recently, Banks and Martin
\cite{bm} conjectured that $f(\N_1)>f(\N_2)>f(N_3)>\cdots$\,.  The inequality
$f(\N_2)>f(\N_3)$ was just established by Bayless, Kinlaw, and Klyve \cite{BKK}.
We prove the following result.
\begin{theorem}
\label{thm:Nk}
There is a positive constant $c$ such that $f(\N_k)\ge c$ for all $k$.
\end{theorem}

We let the letters $p,q,r$ represent primes.  In addition, we let $p_n$ represent the
$n$th prime.  For an integer $a>1$, we let $P(a)$ and $p(a)$ denote the largest and smallest prime factors of $a$. Modifying the notation introduced in \cite{ez}, for a primitive set $A$ let
\begin{align*}
 A_p & = \{a\in A: p(a)\ge p\},\\
 A'_p & = \{a\in A: p(a) = p\},\\
 A''_p & = \{a/p : a\in A'_p\}.
\end{align*}
We let $f(a)=1/(a\log a)$ and so $f(A)=\sum_{a\in A}f(a)$.  
In this language, Zhang's full result \cite{zhang2} states that $f((\N_k)'_p)\le f(p)$ for all primes $p$, $k\ge1$.
We also, let 
$$
g(a)=\frac1a\prod_{p<P(a)}\left(1-\frac1p\right),\quad h(a)=\frac1{a\log P(a)},
$$
with
$ g(A)=\sum_{a\in A}g(a)$ and $h(A)=\sum_{a\in A}h(a)$.

\section{The Erd\H os approach}

In this section we will prove Theorem \ref{thm:egamma}.
We begin with an argument inspired by the original 1935 paper of Erd\H os \cite{erdos35}.

\begin{proposition}\label{lem:erdos}
For any primitive set $A$, if $q\notin A$ then 
$$f(A'_q) < e^\gamma g(q) = \frac{e^\gamma}{q}\prod_{p<q}\bigg(1-\frac{1}{p}\bigg).$$
\end{proposition}
\begin{proof}
For each $a\in A'_q$, let $S_a = \{ba : p(b) \ge P(a)\}$. Note that $S_a$ has asymptotic density $g(a)$. Since $A'_q$ is primitive, we see that the sets $S_a$ are pairwise disjoint. Further, the union of the sets $S_a$ is contained in the set of all natural numbers $m$ with $p(m) = q$, which has asymptotic density $g(q)$. Thus, the sum of densities for each $S_a$ is dominated by $g(q)$, that is,
\begin{align}\label{eq:g(A)}
g(A'_q)=\sum_{a\in A'_q}g(a) \le g(q).
\end{align}

By Theorem 7 in \cite{RS1}, we have for $x\ge285$,
\begin{equation}
\label{eq:mert}
\prod_{p\le x}\left(1-\frac1p\right)>\frac{1}{e^\gamma\log(2x)},
\end{equation}
which may be extended to all $x\ge 1$ by a calculation. Thus, since each $a \in A'_q$ is composite,
$$g(a)=\frac{1}{a}\prod_{p<P(a)}\bigg(1-\frac{1}{p}\bigg) > \frac{e^{-\gamma}}{ a\log\big(2P(a)\big)} > \frac{e^{-\gamma}}{ a\log a} = e^{-\gamma}f(a).$$
Hence by \eqref{eq:g(A)},
\begin{equation*}
    f(A'_q)/e^\gamma < g(A'_q) \le g(q).
\end{equation*}
\end{proof}
\begin{remark}
Let $\sigma$ denote the sum-of-divisors function and let $A$ be the set of $n$ with
$\sigma(n)/n\ge2$ and $\sigma(d)/d<2$ for all proper divisors $d$ of $n$, the set of
primitive nondeficient numbers.  Then an appropriate analog of
 $g(A)$ gives the density of nondeficient numbers, recently shown in \cite{mits1} to lie in the tight interval
$(0.2476171,\,0.2476475)$.
In \cite{JDLpnd}, an analog of 
Proposition \ref{lem:erdos} is a key ingredient for sharp bounds on the reciprocal sum of the
primitive nondeficient numbers.
\end{remark}
\begin{remark}
We have $g(\mathcal P)=1$. It is easy to see by induction over primes $r$ that
\begin{equation*}
\sum_{p\le r}g(p)=\sum_{p\le r}\frac1p\prod_{q<p}\left(1-\frac1q\right)=1-\prod_{p\le r}\left(1-\frac1p\right).
\end{equation*}
Letting $r\to\infty$ we get that $g(\mathcal P)=1$.  There is also a holistic way of seeing this.
Since $g(p)$ is the density of the set of integers with least prime factor $p$,
it would make sense that
$g(\mathcal P)$ is the density of the set of integers which have a least prime factor, which is 1.
To make this rigorous, one notes that the density of the set of integers whose least prime factor
is $>y$ tends to 0 as $y\to\infty$.
As a consequence of $g(\mathcal P)=1$, we have
\begin{equation}
\label{eq:ident}
\sum_{p>2}g(p)=\frac12,
\end{equation}
an identity we will find to be useful.
\end{remark}

For a primitive set $A$, let
$$A^k = \{a: 2^k\| a\in A\}, \qquad B^k = \{a/2^k: a\in A^k\}.$$
The next result will help us prove Theorem \ref{thm:egamma}.

\begin{lemma}\label{lem:egamma}
For a primitive set $A$, let $k\ge1$ be such that $2^k\notin A$. Then we have
$$f(A^{k}) < \frac{e^\gamma}{2^k}\sum_{p\notin A\atop p>2}g(p).$$
\end{lemma}
\begin{proof}
If $2^kp\notin A$ for a prime $p>2$, then $(B^k)'_p$ is a primitive set of odd composite numbers,
 so by Proposition \ref{lem:erdos}, $f((B^k)'_p) < e^\gamma g(p)$.

Now if $2^kp\in A$ for some odd prime $p$, then $(B^{k})'_p=\{p\}$ and note $p\notin A$ by primitivity. We have $f(2^kp) < 2^{-k}e^\gamma g(p)$ since
$$
\frac1{2^kp\log(2^kp)}\le\frac1{2^kp\log(2p)}<\frac{e^\gamma}{2^k}g(p),
$$
which follows from \eqref{eq:mert}.  Hence combining the two cases,
\begin{align*}
f(A^k)=\sum_{p\notin A\atop p>2}f(2^k{\cdot}(B^k)'_p)& \le \sum_{p\in B^k,p\notin A\atop p>2}f(2^kp) + 2^{-k}\sum_{p\notin B^k,p\notin A\atop p>2}f((B^k)'_p)\\
& < \frac{e^\gamma}{2^k}\sum_{p\notin A\atop p>2}g(p).
\end{align*}
\end{proof}
With Lemma \ref{lem:egamma} in hand, we prove $f(A)<e^\gamma$.
\begin{proof}[Proof of Theorem \ref{thm:egamma}]
 From Erd\H os--Zhang \cite{ez}, we have that $f(A_3)<0.92$.  If $2\in A$, then $A'_2=\{2\}$, so that
 $f(A)=f(A_3)+f(A'_2)<0.92+1/(2\log2)<e^\gamma$.  Hence we may assume that $2\notin A$.
 If $A$ contains every odd prime, then $f(A'_2)$ consists of at most one power
 of 2, and the calculation just concluded shows we may assume this is not the case.
Hence there is at least one odd prime $p_0\notin A$.
By Proposition \ref{lem:erdos}, we have
\begin{align}\label{eq:A}
    f(A) &= \sum_pf(A'_p)= \sum_{p\in A}f(p) + \sum_{p\notin A} f(A'_p) < \sum_{p\in A}f(p) + e^\gamma\sum_{p\notin A\atop p>2} g(p) + f(A'_2).
\end{align}
First suppose $A$ contains no powers of $2$. Then by Lemma \ref{lem:egamma},
\begin{align*}
    f(A'_2) = \sum_{k\ge1}f(A^{k}) < \sum_{k\ge1}\frac{e^\gamma}{2^k}\sum_{p\notin A\atop p>2}g(p) = e^\gamma\sum_{p\notin A\atop p>2}g(p).
\end{align*}
Substituting into  \eqref{eq:A}, we conclude, using \eqref{eq:ident},
\begin{align}
\label{eq:fA}
    f(A) & < \sum_{p\in A}f(p) + 2e^\gamma\sum_{p\notin A\atop p>2}g(p) \le 2e^\gamma\sum_{p>2}g(p) = e^\gamma.
\end{align}
For the last inequality we used that for every prime $p$,
\begin{equation}
\label{eq:fgineq2}
\frac{f(p)}{e^\gamma g(p)}<1.082,
\end{equation}
which follows after a short calculation using \cite[Theorem 7]{RS1}.

Now if $2^K\in A$ for some positive integer $K$, then $K$ is unique and $K\ge 2$.
Also $A^K=\{2^K\}$ and
 $A^k=\emptyset$ for all $k>K$, so again by Lemma \ref{lem:egamma},
\begin{align*}
    f(A'_2) = \sum_{k=1}^Kf(A^{k}) = f(2^K) + \sum_{k=1}^{K-1}\frac{e^\gamma}{2^k}\sum_{p\notin A\atop p>2}g(p) = f(2^K) + (1-2^{1-K})e^\gamma\sum_{p\notin A\atop p>2}g(p).
\end{align*}
Substituting into \eqref{eq:A} gives
\begin{align}
\label{eq:fAK}
    f(A) < \sum_{p\in A}f(p) + f(2^K) + (2-2^{1-K})e^\gamma\sum_{p\notin A\atop p>2}g(p) & \le f(2^K) + (2-2^{-1})e^\gamma\sum_{p>2}g(p)\nonumber\\
    & \le f(2^2) + (1-2^{-2})e^\gamma < e^\gamma,
\end{align}
using $K\ge2$, the identity \eqref{eq:ident}, inequality \eqref{eq:fgineq2}, and $f(2^2)< 2^{-2} e^\gamma$. This completes the proof.
\end{proof}

%
%

\section{Mertens primes}

In this section we will prove Theorems \ref{thm:race} and Theorem \ref{thm:support}.
Note that by Mertens' theorem,
$$
\prod_{p<x}\left(1-\frac1p\right)\sim\frac1{e^\gamma\log x},\quad x\to\infty,
$$
where $\gamma$ is Euler's constant.
We say a prime $q$ is {\bf Mertens} if 
\begin{equation}
\label{eq:help}
e^\gamma\prod_{p<q}\Big(1-\frac{1}{p}\Big) \le \frac{1}{\log q},
\end{equation}
and let $\mathcal P^{\textrm{Mert}}$ denote the set of Mertens primes. We are interested in Mertens primes because of the following consequence of Proposition \ref{lem:erdos}, which shows that every Mertens prime is Erd\H os strong.

\begin{corollary}\label{cor:mert}
Let $A$ be a primitive set.
If $q\in \PP^{\rm Mert}$, then $f(A'_q)\le f(q)$. Hence if $A'_q \subset \{q\}$ for all $q\notin \PP^{\rm Mert}$, then $A$ satisfies the Erd\H os conjecture.
\end{corollary}
\begin{proof}
By Proposition \ref{lem:erdos} we have $f(A'_q)\le\max\{e^\gamma g(q),f(q)\}$. If $q\in\PP^{\textrm{Mert}}$, then 
$$e^\gamma g(q) = \frac{e^\gamma}{q}\prod_{p<q}\bigg(1-\frac{1}{p}\bigg)\le \frac{1}{q\log q} =f(q),$$
so $f(A'_q)\le f(q)$.
\end{proof}

Now, one would hope that the Mertens inequality \eqref{eq:help}
holds for all primes $q$. However, \eqref{eq:help} fails for $q=2$ since $e^\gamma > 1/\log 2$. We have computed that $q$ is indeed a Mertens prime for all $2<q\le p_{10^8} = 2{,}038{,}074{,}743$, thus proving the unconditional part of Theorem \ref{thm:race}.

\subsection{Proof of Theorem \ref{thm:race}}

To complete the proof, we use a result of Lamzouri \cite{lamz} relating the Mertens inequality to the race between $\pi(x)$
and $\li(x)$, studied by Rubinstein and Sarnak \cite{RubSarn}. Under the assumption of RH and LI,
he proved that the set $\mathcal N$ of real numbers $x$ satisfying
\begin{align*}
    e^\gamma \prod_{p\le x}\bigg(1-\frac{1}{p}\bigg) > \frac{1}{\log x},
\end{align*}
has logarithmic density $\delta(\mathcal{N})$ equal to the logarithmic density of numbers $x$ with $\pi(x)>\li(x)$,
and in particular
\begin{align}
\delta(\mathcal N) = \lim_{x\to\infty}\frac{1}{\log x}\int_{t\in \mathcal N\cap[2,x]}\frac{dt}{t} = 0.00000026\ldots\,.
\end{align}

We note that if a prime $p=p_n\in \mathcal N$, then for $p'=p_{n+1}$ we have $[p,p')\subset\mathcal N$ because the prime product on the left-hand side is constant on $[p,p')$, while $1/\log x$ is decreasing for $x\in [p,p')$. 

The set of primes $\QQ$ in $\mathcal N$ is precisely the set of non-Mertens primes, so $\QQ=\mathcal P\setminus\mathcal P^{\textrm{Mert}}$.  From the above observation, we may leverage knowledge of the continuous logarithmic density $\delta(\mathcal N)$ to obtain an upper bound on the relative (upper) logarithmic density of non-Mertens primes
\begin{align}
\label{eq:dens}
\bar\delta(\QQ) := \limsup_{x\to \infty}\frac{1}{\log x}\sum_{p\le x\atop p\in \QQ}\frac{\log p}{p}.
\end{align}

 From the above observation, we have
\begin{align*}
\delta(\mathcal N) \ge \limsup_{x\to\infty}\frac{1}{\log x}\sum_{p\le x\atop p\in \mathcal Q}\int_{p}^{p'}\frac{dt}{t} & = \limsup_{x\to\infty}\frac{1}{\log x}\sum_{p\le x\atop p\in \mathcal Q}\log(p'/p).
\end{align*}
Then letting $d_p=p'-p$ be the gap between consecutive primes, we have
\begin{align*}
\delta(\mathcal N) \ge \limsup_{x\to\infty}\frac{1}{\log x}\sum_{p\le x\atop p\in \mathcal Q}\frac{d_p}{p},
\end{align*}
since $\sum\log(p'/p) = \sum d_p/p + O(1)$. The average gap is roughly $\log p$, so we may consider the primes for which $d_p < \epsilon \log p$, for a small positive constant $\epsilon$ to be determined.

We claim
\begin{align}\label{eq:RVclaim}
\limsup_{x\to\infty}\frac1{\log x}\sum_{\substack{p\le x\\ d_p < \epsilon \log p}}\frac{\log p}{p} \ \le \ 16\epsilon,
\end{align}
from which it follows
\begin{align*}
\bar\delta(\QQ) & = \limsup_{x\to\infty}\frac{1}{\log x}\sum_{p\le x\atop p\in \QQ}\frac{\log p}{p} \le \limsup_{x\to\infty}\frac{1}{\log x}\Big(\sum_{\substack{p\le x\\ p\in \QQ\\d_p \ge \epsilon \log p}}\frac{d_p/\epsilon}{p} + \sum_{\substack{p\le x\\ d_p < \epsilon \log p}}\frac{\log p}{p}\Big)\\
& \le \delta(\mathcal N)/\epsilon + 16\epsilon.
\end{align*}
Hence to prove Theorem \ref{thm:race} it suffices to prove \eqref{eq:RVclaim}, since taking $\epsilon = \sqrt{\delta(\mathcal N)}/4$ gives
\begin{align}
    \bar\delta(\QQ) < 8\sqrt{\delta(\mathcal N)} < 4.2\times10^{-3}.
\end{align}

By Riesel-Vaughan \cite[Lemma 5]{RV}, the number of primes $p$ up to $x$ with $p+d$ also prime is at most
\begin{align*}
    \sum_{p\le x\atop p+d\textrm{ prime}}1 \le \frac{8c_2x}{\log^2 x}\prod_{p\mid d\atop p>2}\frac{p-1}{p-2},
\end{align*}
where $c_2$ is for the twin-prime constant $2\prod_{p>2}p(p-2)/(p-1)^2=1.3203\ldots$. 
Denote the prime product by $F(d) = \prod_{p\mid d\atop p>2}\frac{p-1}{p-2}$, and consider the multiplicative function $H(d) = \sum_{u\mid d}\mu(u)F(d/u)$. We have $H(2^k)=0$ for all $k\ge1$, and for $p>2$ we have $H(p)=F(p)-1$, and $H(p^k)=0$ if $k\ge2$. Thus,
\begin{align*}
\sum_{d\le y} F(d) & = \sum_{d\le y}\sum_{u\mid d}H(u) = \sum_{u\le y}H(u)\sum_{d\le y/u}1\le y\sum_{u\le y}\frac{H(u)}{u} \le y\prod_{p>2}\Big(1 + \frac{H(p)}{p}\Big)\\
& = y\prod_{p>2}\Big(1 + \frac{(p-1)/(p-2)-1}{p}\Big) = y\prod_{p>2}\Big(1 + \frac{1}{p(p-2)}\Big).
\end{align*}
Noting that $c_2':=\prod_{p>2}(1 + 1/[p(p-2)])=2/c_2$, we have
\begin{align*}
\sum_{\substack{p\le x\\ d_p < \epsilon \log p}}1\le \sum_{d\le \epsilon\log x}\sum_{p\le x\atop p+d\textrm{ prime}}1 \le \frac{8c_2x}{\log^2 x}\sum_{d\le \epsilon\log x}F(d) \le \epsilon\frac{8c_2c_2'x}{\log x} = \epsilon\frac{16x}{\log x}.
\end{align*}
Thus, \eqref{eq:RVclaim} now follows by partial summation, and the proof is complete.

\begin{remark}
The concept of relative upper logarithmic density of the set of non-Mertens primes in
\eqref{eq:dens}  can be replaced in the theorem with
$$
\bar\delta_0(\QQ):=\limsup_{x\to\infty}\frac1{\log\log x}\sum_{\substack{p\le x\\p\in \QQ}}\frac1p.
$$
Indeed, $\bar\delta_0(\QQ)\le\bar\delta(\QQ)$ follows from the identity
$$
\sum_{\substack{p\le x\\p\in \QQ}}\frac1p=\frac1{\log x}\sum_{\substack{p\le x\\p\in \QQ}}\frac{\log p}p
+\int_2^x\frac1{t(\log t)^2}\sum_{\substack{p\le t\\p\in \QQ}}\frac{\log p}p\,dt.
$$
\end{remark}

\begin{remark}
\label{rmk:martin}
Greg Martin has indicated to us that one should be able to prove (under RH and LI) that the relative logarithmic density
of $\QQ$ exists and is equal to the logarithmic density of $\mathcal N$.  The idea is as follows.  Partition
the positive reals into intervals of the form $[y,y+y^{1/3})$.  Let $E_1$ be the union of those intervals
$[y,y+y^{1/3})$ where the sign of $e^\gamma\prod_{p\le x}(1-1/p)-1/\log x$ is not constant and let
$E_2$ be the union of those intervals $[y,y+y^{1/3})$ which do not have $\sim y^{1/3}/\log y$ primes
as $y\to\infty$.  The the logarithmic density of $E_1\cup E_2$ can be shown to be 0, from which the assertion 
follows.
\end{remark}

\subsection{Proof of Theorem \ref{thm:support}}

We now use some numerical estimates of Dusart \cite{dusart} to prove Theorem \ref{thm:support}.

We say a pair of primes $p\le q$ is a {\bf Mertens pair} if
$$
\prod_{p\le r<q}\left(1-\frac1r\right)>\frac{\log p}{\log pq}.
$$
We claim that every pair of primes $p,q$ with $2<p\le q<e^{10^6}$ is a Mertens pair.
Assume this and let $A$ be a primitive set supported on the odd primes to $e^{10^6}$.
By \eqref{eq:g(A)}, if $p\notin A$, we have
\begin{align*}
\frac1{p}&\ge\sum_{a\in A'_p}\frac1a\prod_{p\le r <P(a)}\left(1-\frac1{r}\right)
>\sum_{a\in A'_p}\frac{\log p}{a\log(p\,P(a))}\\
&\ge\sum_{a\in A'_p}\frac{\log p}{a\log a}=f(A'_p)\log p.
\end{align*}
Dividing by $\log p$ we obtain $f(A'_p)\le f(p)$, which also holds if $p\in A$. Thus, the claim about Mertens pairs implies the theorem.

To prove the claim, first note that if $p$ is
a Mertens prime, then $p,q$ is a Mertens pair for all primes $q\ge p$.  Indeed, we have
$$
\prod_{p\le r<q}\left(1-\frac1r\right)=\prod_{r<p}\left(1-\frac1r\right)^{-1}\prod_{r<q}\left(1-\frac1r\right)
>e^\gamma \log p\prod_{r<q}\left(1-\frac1r\right).
$$
By \eqref{eq:mert}, this last product exceeds $e^{-\gamma}/\log(2q)>e^{-\gamma}/\log(pq)$,
and using this in the above display shows that $p,q$ is indeed a Mertens pair.  Since
all of the odd primes up to $p_{10^8}$ are Mertens, to complete the proof of our assertion,
it suffices to consider the case when $p>p_{10^8}$.  Define $E_p$ via the
equation
$$
\prod_{r<p}\left(1-\frac1r\right)=\frac{1+E_p}{e^\gamma \log p}.
$$
Using \cite[Theorem 5.9]{dusart}, we have for $p>2{,}278{,}382$,
\begin{equation}
\label{eq:D}
|E_p|\le .2/(\log p)^3.
\end{equation}
  A routine calculation
shows that if $p\le q<e^{4.999(\log p)^4}$, then
$$
\prod_{p\le r<q}\left(1-\frac1r\right)=\frac{\log p}{\log q}\cdot\frac{1+E_q}{1+E_p}
>\frac{\log p}{\log pq}.
$$
It remains to note that $4.999(\log p_{10^8})^4 >1{,}055{,}356$.

It seems interesting to record the principle that we used in the proof.
\begin{corollary}
\label{cor:mertens}
If $A$ is a primitive set such that $p(a),P(a)$ is a Mertens pair for each $a\in A$, then
$f(A)\le f(\PP(A))$.
\end{corollary}

\begin{remark}
\label{rmk:ford}
Kevin Ford has noted to us the remarkable similarity between the concept of Mertens primes in this
paper and the numbers 
$$
\gamma_n=\left(\gamma+\sum_{k\le n}\frac{\log p_k}{p_k-1}\right)\prod_{k\le n}\left(1-\frac1{p_k}\right)
$$
discussed in Diamond--Ford \cite{DF}.  In particular, while it may not be obvious from the definition, the analysis in 
\cite{DF} on whether the sequence $\gamma_1,\gamma_2,\dots$ is monotone
 is quite similar to the analysis in \cite{lamz} on the Mertens inequality. Though the numerical evidence
 seems to indicate we always have $\gamma_{n+1}<\gamma_n$, this is disproved in \cite{DF}, and it is
 indicated there that the first time this fails may be near $1.9\cdot10^{215}$.  This may also be near where the
 first odd non-Mertens prime exists.  If this is the case, and under assumption of RH, it may be that
 every pair of primes $p\le q$ is a Mertens pair when $p>2$ and $q < \exp(10^{100})$.
 \end{remark}

\section{Odd primitive sets}

In this section we prove Theorem \ref{thm:no8s} and establish a curious result on parity for primitive sets.

Let
$$
\epsilon_0=\sum_{\substack{p>2\\p\notin\PP^{\rm Mert}}}\left(e^\gamma g(p)-f(p)\right).
$$
\begin{lemma}
\label{lem:eps}
We have $0\le\epsilon_0<2.37\times10^{-7}$.
\end{lemma}
\begin{proof}
By the definition of $\PP^{\rm Mert}$, the summands in the definition of $\epsilon_0$ are nonnegative,
so that $\epsilon_0\ge0$.    If $p>2$ is not Mertens, then $p>p_{10^8}>2\times10^9$, so that \eqref{eq:D} shows that
\begin{equation}
\label{eq:nonM}
e^\gamma g(p)-f(p)<\frac1{5p(\log p)^4}.
\end{equation}
By \cite[Proposition 5.16]{dusart}, we have
$$
p_n  >n(\log n+\log\log n-1 +(\log\log n-2.1)/\log n,\quad n\ge 2.
$$
Using this we find that 
$$
\sum_{n>10^8}\frac1{5p_n(\log p_n)^4}<2.37\times10^{-7},
$$
which with \eqref{eq:nonM} completes the proof.
\end{proof}

\begin{remark}
Clearly, a smaller bound for $\epsilon_0$ would follow by raising the search limit for Mertens primes.
Another small improvement could be made using the estimate in \cite{axler} for $p_n$. 
It follows from the ideas in Remark \ref{rmk:martin} that $\epsilon_0>0$.  Further, it may be provable
 from the ideas in
Remark \ref{rmk:ford} that $\epsilon_0<10^{-100}$ if the Riemann Hypothesis holds.
\end{remark}

We have the following result.
\begin{theorem}
\label{thm:odd}
For any odd primitive set $A$, we have
\begin{align}
\label{eq:odd}
f(A)  \le f(\PP(A))+\epsilon_0.
\end{align}
\end{theorem}
\begin{proof}
Assume that $A$ is an odd primitive set.  We have
$$
f(A)=\sum_{p\in\PP(A)}f(A'_p)\le\sum_{p\in\PP(A)\cap\PP^{\rm Mert}}f(p)+\sum_{p\in\PP(A)\setminus\PP^{\rm Mert}}e^\gamma g(p)
\le\epsilon_0+\sum_{p\in\PP(A)}f(p)
$$
by the definition of $\epsilon_0$.  This completes the proof.
\end{proof}

This theorem yields the following corollary.
\begin{corollary}\label{cor:8}
If $A$ is a primitive set containing no multiple of $8$, then \eqref{eq:odd} holds.
\end{corollary}
\begin{proof}
We have seen the corollary in the case that $A$ is odd.
  Next, suppose that $A$ contains an even number, but no multiple of 4.
If $2\in A$, the result follows by applying Theorem \ref{thm:odd} to $A\setminus\{2\}$, so assume
$2\notin A$.  Then $A''_2$ is an odd primitive set and $f(A'_2)\le f(A''_2)/2$.  
We have by the odd case that
\begin{equation}
\label{eq:2mod4}
f(A)=f(A_3)+f(A'_2)< f(\PP(A_3))+\epsilon_0+\frac12\left(f(\PP(A''_2))+\epsilon_0\right).
\end{equation}
Since 
$$
\frac12f(\PP(A''_2))\le\frac12f(\PP\setminus\{2\})<0.4577
$$
and $f(2)=0.7213\dots$, \eqref{eq:2mod4} and Lemma \ref{lem:eps} imply that $f(A)<f(\PP(A))$,
which is stronger than required.  The case when $A$ contains a multiple of 4
but no multiple of 8 follows in a similar fashion.
\end{proof}
Since a cube-free number cannot be divisible by 8,  \eqref{eq:odd} holds for all primitive sets $A$ of cube-free numbers.  Also, the proof of Corollary \ref{cor:8} can be adapted to show that \eqref{eq:odd} holds
for all primitive sets $A$ containing no number that is 4~(mod~8).

We close out this section with a curious result about those primitive sets $A$ where \eqref{eq:odd}
does not hold.
Namely, the Erd\H os conjecture must then hold for the set of odd members of $A$.
 Put another way, \eqref{eq:odd} holds for any primitive set $A$ for
which the Erd\H os conjecture for the odd members of $A$ {\it fails}.
\begin{theorem}
\label{thm:curious}
If $A$ is a primitive set with $f(A)> f(\PP(A))+\epsilon_0$, then $f(A_3)<f(\PP(A_3))$.
\end{theorem}
\begin{proof}[Proof (Sketch)]
Without loss of generality, we may include in $A$ all primes not in $\PP(A)$, and so assume that $\PP(A)=\PP$
and $f(A)>C+\epsilon_0$.
By Theorem \ref{thm:odd} we may assume that $A$ is not odd, and by Corollary \ref{cor:8}
we may assume that $2\notin A$.  By the proof of Theorem \ref{thm:egamma} (see \eqref{eq:fA} and 
\eqref{eq:fAK}), if $3\in A$, we have
$$
f(A)<f(3)+\frac23e^\gamma<C,
$$
a contradiction, so we may assume that $3\notin A$.  We now apply the method of proof of Theorem \ref{thm:egamma}
to $A_3$, where powers of 3 replace powers of 2.  This leads to
$$
f(A_3)<\frac12e^\gamma<C-f(2)=f(\PP(A_3)).
$$
This completes the argument.
\end{proof}

\section{Zhang primes and the Banks--Martin conjecture}

Note that
$$
\sum_{p\ge x}\frac1{p\log p}\sim\frac1{\log x},\quad x\to\infty.
$$
In Erd\H os--Zhang \cite{ez} and in Zhang \cite{zhang2}, numerical approximations
to this asymptotic relation are exploited.
Say a prime $q$  is {\bf Zhang} if 
$$\sum_{p\ge q}\frac{1}{p\log p} \le \frac{1}{\log q}.$$
Let $\mathcal P^{\textrm{Zh}}$ denote the set of Zhang primes. We are interested in Zhang primes because of the following result.

\begin{theorem}\label{thm:zhang}
If $\mathcal P(A'_p)\subset \mathcal P^{\textrm{Zh}}$, then $f(A'_p) \le f(p)$. 
Hence the Erd\H os conjecture holds for all primitive sets $A$ supported on $\mathcal P^{\textrm{Zh}}$.
\end{theorem}
\begin{proof}
As in \cite{ez} it suffices to prove the theorem in the case that $A$ is a finite set.  By $d^\circ(A)$ we mean
the maximal value of $\Omega(a)$ for $a\in A$.
We proceed by induction on $d^\circ(A_p')$. If $d^\circ(A'_p)\le 1$, then $f(A_p') \le  f(p)$. If $d^\circ(A'_p)> 1$, then $f(A_p')\le f(A_p'')/p$. The primitive set $B:=A_p''$ satisfies $f(B)=f(B_p)=\sum_{q\ge p}f(B_q')$. Since $d^\circ(B_q') \le d^\circ(B) < d^\circ(A'_p)$, by induction we have $f(B_q') \le f(q)$. Thus, since $p$ is Zhang,
$$f(A_p'') = f(B)=\sum_{q\ge p}f(B_q')\le \sum_{q\ge p}\frac{1}{q\log q} \le \frac{1}{\log p},$$
from which we obtain $f(A_p')\le f(A_p'')/p \le 1/(p\log p)$. This completes the proof.
\end{proof}

 From this one might hope that all primes are Zhang. However, the prime 2 is not Zhang since $C> 1/\log 2$,
 and the prime 3 is not Zhang since $C-1/(2\log2)>1/\log3$. Nevertheless, as with Mertens primes, it is true that the remaining
primes up to $p_{10^8}$ are Zhang. Indeed, starting from \eqref{eq:cohen}, we
computed that
\begin{align}
\sum_{p\ge q}\frac1{p\log p} = C - \sum_{p < q}\frac{1}{p\log p} \le \frac{1}{\log q}\qquad\textrm{for all } 3< q \le p_{10^8}.
\end{align}
The computation stopped at $10^8$ for convenience, and one could likely extend this further with some patience.  It seems likely that there is also a ``race" between $\sum_{p\ge q}1/(p\log p)$
and $1/\log q$, as with Mertens primes, and that a large logarithmic density of primes $q$
are Zhang, with a small logarithmic density of primes failing to be Zhang.


A related conjecture due to Banks and Martin \cite{bm} is the chain of inequalities,
\begin{align*}
    \sum_{p}\frac{1}{p\log p} > \sum_{p\le q}\frac{1}{pq\log pq} > \sum_{p\le q\le r}\frac{1}{pqr\log pqr} > \cdots,
\end{align*}
succinctly written as $f(\N_k) > f(\N_{k+1})$ for all $k\ge1$, where 
$\N_k = \{n: \Omega(n) = k\}$.  As mentioned in the introduction, we know only that
$f(\N_1)> f(\N_k)$ for all $k\ge2$ and $f(\N_2)>f(\N_3)$.
More generally, for a subset $Q$ of primes, let $\N_k(Q)$ denote the subset of $\N_k$ supported on $Q$. A result of Zhang \cite{zhang2} impies that $f(\N_1(Q)) > f(\N_{k}(Q))$ for all $k>1$, while Banks and Martin showed that $f(\N_k(Q)) > f(\N_{k+1}(Q))$ if $\sum_{p\in Q}1/p$ is not too large.
We prove a similar result in the case where $Q$ is a subset of the Zhang primes and we
replace $f(\N_k(Q))$ with $h(\N_k(Q))$. Recall $h(A) = \sum_{a\in A}1/(a\log P(a))$.

\begin{proposition}
For all $k\ge1$ and $Q\subset \mathcal P^{\textrm{Zh}}$,
we have $h(\N_k(Q)) \ge h(\N_{k+1}(Q))$.
 \end{proposition}
\begin{proof}
Since $p_k$ is a Zhang prime, we have
\begin{align*}
h(\N_{k+1}(Q)) & = \sum_{\substack{q_1\le \cdots\le q_{k+1}\\ q_i\in Q}}\frac{1}{q_1\cdots q_k q_{k+1}\log q_{k+1}} \\
& = \sum_{\substack{q_1\le \cdots\le q_{k}\\ q_i\in Q}}\frac{1}{q_1\cdots q_{k}}\sum_{q_{k+1}\ge q_{k}}\frac{1}{q_{k+1}\log q_{k+1}}\\
& \le \sum_{\substack{q_1\le \cdots\le q_{k}\\ q_i\in Q}}\frac{1}{q_1\cdots q_{k} \log q_k} = h(\N_{k}(Q)).
\end{align*}
This completes the proof.
\end{proof}

It is interesting that if we do not in some way restrict the primes used, the analogue 
of the Banks--Martin conjecture for the function $h$ fails.  In particular, we have
$$
h(\N_2)>\sum_{m\le 10^4}\frac1{p_m}\sum_{n\ge m}\frac1{p_n\log p_n}
=\sum_{m\le 10^4}\frac1{p_m}\left(C-\sum_{k<m}\frac1{p_k\log p_k}\right)
>1.638,
$$
while $h(\N_1)=C<1.637$.

It is also interesting that the analogue of the Banks--Martin conjecture for the function $g$ is
false since
$$
1=g(\N_1)=g(\N_2)=g(\N_3)=\cdots\,.
$$
We have already shown in \eqref{eq:g(A)} that
$g(A'_q)\le g(q)$ for any primitive set $A$ and prime $q$, so the analogue for $g$ of the strong Erd\H os conjecture holds.

\subsection{Proof of Theorem \ref{thm:Nk}.}

We now return to the function $f$ and prove Theorem \ref{thm:Nk}.

We may assume that $k$ is large.  Let $m=\lfloor\sqrt{k}\rfloor$ and let
$B(n)=e^{e^n}$.  We have
\begin{align*}
f(\N_k)&=\sum_{\Omega(a)=k}\frac1{a\log a} >\sum_{\substack{\Omega(a)=k\\
e^{e^{k}}<a\le e^{e^{k+m}}}}\frac1{a\log a}\\
&=\sum_{j\le m}\sum_{\substack{\Omega(a)=k\\B({k+j-1})<a\le
B(k+j)}}\frac1{a\log a} > 
\sum_{j\le m}\frac1{\log
B({k+j})}\sum_{\substack{\Omega(a)=k\\B(k+j-1)<a\le B({k+j})}}
\frac1a.
\end{align*}
Thus it suffices to show that there is a positive constant $c$ such that for $j\le m$ we have
\begin{equation}
\label{eq:ss}
\sum_{\substack{\Omega(a)=k\\B({k+j-1})<a\le B({k+j})}}\frac1a
\ge c\frac{\log B({k+j})}m=c\frac{e^{k+j}}m,
\end{equation}
so that the proposition will follow.

Let $N_k(x)$ denote the number of members of $\N_k$ in $[1,x]$.
We use the Sathe--Selberg theorem, see \cite[Theorem 7.19]{MV},
from which we have that uniformly for $B({k})< x\le B({k+m})$, as $k\to\infty$,
$$
N_k(x)\sim \frac x{k!}\frac{(\log\log x)^k}{\log x}.
$$
This result also follows from Erd\H os \cite{erdos48}.

We have
\begin{align*}
\sum_{\substack{\Omega(a)=k\\B({k+j-1})<a\le B({k+j})}}\frac1a
&>\int_{B({k+j-1})}^{B({k+j})}\frac{N_k(x)-N_k(B({k+j-1}))}{x^2}\,dx\\
&\gg \int_{2B({k+j-1})}^{B({k+j})}\frac{N_k(x)}{x^2}\,dx.
\end{align*}
Thus,
\begin{align*}
\sum_{\substack{\Omega(a)=k\\B({k+j-1})<a\le B({k+j} )}}\frac1a
&\gg \frac{(\log\log
B({k+j-1}))^k}{k!}\int_{2B({k+j-1})}^{B({k+j})}\frac{dx}{x\log x}\\
&=\frac{(k+j-1)^k}{k!}(\log\log B({k+j})-\log\log(2B({k+j-1}))\\
&\gg\frac{(k+j-1)^k}{k!}\gg\frac{e^{k+j}}{\sqrt{k}},
\end{align*}
the last estimate following from Stirling's formula.  This proves
\eqref{eq:ss},
and so the theorem.

The sets $\N_k$ and Theorem \ref{thm:Nk} give us the following result.
\begin{corollary}
\label{cor:largex}
We have that
$$
\limsup_{x\to\infty}\{f(A):A\subset[x,\infty),\,A \textnormal{ primitive}\}>0.
$$
\end{corollary}


\section*{Acknowledgments}
We thank Greg Martin for  the content of Remark \ref{rmk:martin} and Kevin Ford for the content of
Remark \ref{rmk:ford}.
We thank Paul Kinlaw and Zhenxiang Zhang for some helpful comments.

\bibliographystyle{amsplain}

\end{document}